\documentclass[12pt]{amsart}
\usepackage[margin=1.25in]{geometry}

\usepackage{amsmath}

\usepackage{amssymb}
\usepackage{amscd}
\usepackage{enumerate}

\usepackage[numbers,sort&compress]{natbib}
\bibpunct[, ]{[}{]}{,}{n}{,}{,}

\theoremstyle{plain}
\newtheorem{theorem}{Theorem}[section]
\newtheorem{lemma}[theorem]{Lemma}

\newtheorem{claim}[theorem]{Claim}

\theoremstyle{definition}

\newtheorem{exa}[theorem]{Example}
\newtheorem{con}[theorem]{Conjecture}
\newtheorem{question}[theorem]{Question}

\theoremstyle{remark}

\newcommand{\N}{\mathbb{N}}
\newcommand{\Z}{\mathbb{Z}}
\newcommand{\R}{\mathbb{R}}
\newcommand{\F}{\mathcal{F}}
\newcommand{\G}{\mathcal{G}}
\newcommand{\s}{\mathcal{S}}
\newcommand{\h}{\mathcal{H}}
\newcommand{\W}{\mathcal{W}}
\newcommand{\WW}{\mathbf{W}}
\newcommand{\w}{\mathbf{w}}
\DeclareMathOperator{\conv}{conv}
\DeclareMathOperator{\gap}{Gap}
\newcommand{\vre}{\varepsilon}
\newcommand{\dist}{\operatorname{dist}}
\newcommand{\pos}{\operatorname{Pos}}
\newcommand{\ignore}[1]{}

\newcommand{\df}{{\, \stackrel{\mathrm{def}}{=}\, }}

\newenvironment{subproof}[1][\proofname]{%
  \begin{proof}[#1]%
}{%
  \end{proof}%
}

\begin{document}

\title{Attractors of sequences of iterated function systems}

\title{Attractors of sequences of iterated function systems}
\author{Ryan Broderick} 

\maketitle

\begin{abstract}
If $\F$ and $\G$ are iterated function systems, then any infinite word $\WW$ in the symbols $\F$ and $\G$
induces a limit set. It is natural to ask whether this Cantor set can also be
realized as the limit set of a single iterated function system $\h$. We prove that 
if $\F$, $\G$, and $\h$ consist of $C^{1+\alpha}$ diffeomorphisms, then under
some additional constraints on $\F$ and $\G$ the answer is no.
This problem is motivated by the spectral theory of one-dimensional quasicrystals.
\end{abstract}

\section*{Aknowledgements}

The author thanks Anton Gorodetski for introducing him to this problem and for helpful discussions.
The author also wishes to thank the referee, whose careful reading greatly improved the quality
of the paper.

\section{Introduction}
If $D$ is a finite union of closed intervals,
then any family $\F$ of continuous, contracting functions $f_1, \dots, f_k \colon D \to D$
with disjoint images
defines a function on the compact subsets of $D$
via $F(K) = \bigcup_{i=1}^k f_i(K)$.\footnote{Throughout we will denote such a family by a caligraphic letter
and denote the induced function on compact sets by the corresponding standard letter.}
We define $C(\F)$ to be the compact set $\bigcap_{n=0}^\infty F^{n}(D)$.
This is the \emph{dynamically defined Cantor set} associated to $\F$.
We also call $\F$ an \emph{iterated function system}, or \emph{IFS}, and refer to $C(\F)$ as the \emph{limit set} of $\F$.
If $\F$ and $\G$ are two IFSs and $\mathbf{W} = (\W_n)_{n\in \N} \in \{\F, \G\}^{\N}$,
then we write $C(\WW) = \bigcap_{n=1}^\infty (W_1 \circ \dots \circ W_n)(D)$.
We will also write $C_n(\WW) = (W_1 \circ \dots \circ W_n)(D)$
and $C_n(\F) = F^n(D)$.
We say that $\F$ is a $C^{1+\alpha}$-diffeomorphic IFS if every function in $\F$ is 
a $C^{1+\alpha}$ diffeomorphism.

One-dimensional IFSs are well-studied, but relatively little is known about the Cantor sets
that may arise as their attractors. In \cite{McD}, it is shown that for a $C^1$-diffeomorphism IFS on the
circle, ratios of gap lengths accumulate at $1$, so for example the ternary Cantor set cannot arise.
For $C^{1+\alpha}$-diffeomorphic IFSs, it is known that the attractor must have Lebesgue measure $0$,
whereas this does not hold in the $C^1$ case. (See \cite{B}.)
See also \cite{DKN1, DKN2} for further results.
We are interested in the following question: given $\F$, $\G$, and $\WW \in \{\F,\G\}^{\N}$, does
there exist an IFS $\h$
such that 
$C(\WW) = C(\h)$?

This question was initially motivated by the spectral theory of one-dimensional quasicrystals.
Namely, consider the discrete Schr{\" o}dinger operator $H \colon \ell^2(\Z) \to \ell^2(\Z)$ given by
\begin{equation}
\label{Schr-operator}
(Hu)(n) = w_{n+1} u_{n+1} + w_n u_{n-1},
\end{equation}
where the sequence $(w_n)_{n\in \Z}$ is called the potential.
This is the off-diagonal case, which ensures that the spectrum of $H$ is symmetric around $0$; see Proposition 2.3 in \cite{Tak}.
See the appendix to \cite{DG} for more details on diagonal and off-diagonal operators.
We will consider the case that $(w_n)$ is a Sturmian sequence, i.e.
\[ 
w_n = \chi_{[0, 1-\alpha)}(\alpha n \text{ (mod } 1)),
\]
where $\alpha \in [0,1)$ is an irrational number and $\lambda > 0$ is called the coupling constant.
It is known that the spectrum of (\ref{Schr-operator}) in this case is a Cantor set of zero Lebesgue measure.

Let $[a_1, a_2, \dots]$ be the continued fraction expansion of $\alpha$, i.e.
\[
\alpha = \dfrac{1}{a_1 + \dfrac{1}{a_2 + \ddots}}.
\]
It is known (see \cite{Mei}, \cite{Gir}) that if the sequence $(a_n)_{n \in \N}$ is eventually periodic, then the spectrum of $H$
is a dynamically defined Cantor set. This has enormous consequences, e.g. see \cite{Mei} for the case $\alpha = [1, 1, 1, \dots]$.
On the other hand, explicit examples when the spectrum of $H$ is not a dynamically defined Cantor set are known \cite{LPW}.
Therefore, it is natural to ask whether eventual periodicity of the partial quotients of $\alpha$ is not only sufficient but also necessary
for the spectrum of $H$ to be a dynamically defined Cantor set.
\begin{con}
The spectrum of $H$ as in \ref{Schr-operator} is a dynamically defined Cantor set if and only if $\alpha$ has an eventually periodic continued fraction expansion.
\end{con}
The results of this paper offer support for this conjecture.

We will show that under certain assumptions on IFSs $\F$ and $\G$, 
we have $C(\WW) = C(\h)$ for some $C^{1+\alpha}$-diffeomorphic system $\h$
only in the trivial
case that $\WW$ is eventually periodic. In the case that $\WW$ is not eventually periodic we will show
that $C(\WW)$ is different from a given $C(\h)$
by considering the gaps of each set. 
Given a compact set $K \subset \R$, a \emph{gap} of $K$ is an interval
$(x,y)$ such that $\{x, y\} \subset K$, $x <  y$, and $(x,y) \cap K = \varnothing$.
In other words, a gap of $K$ is a connected component of $\conv(K) \setminus K$,
where, here and throughout the paper, $\conv(\cdot)$ denotes the convex hull.

Our main theorem concerns the case where $\F$ and $\G$ are symmetric. We say an IFS $\F$ defined on
an interval $I = B(c, r)$ is \emph{symmetric} if for each $f \in \F$ there exists $\overline{f} \in \F$
such that $\overline{f}(x) - c = c - f(x)$. In particular, for each $n$ the set $C_n(\F)$ is symmetric around the center
of $I$.

\begin{theorem}
\label{gen-bdd-deriv}
Let $\F$ and $\G$ be symmetric, $C^{1+\alpha}$-diffeomorphic IFSs defined on the same closed interval $I$.
Suppose $C(\F) \neq C(\G)$ and $f'(x) \in (0,1/2)$ for every $f \in \F \cup \G$ and every $x \in I$.
Let $\WW = (\W_n) \in \{\F,\G\}^\N$
and suppose there exists a $C^{1+\beta}$-diffeomorphic IFS $\h$ such that $C(\WW) = C(\h)$. 
Then $\WW$ is eventually periodic.
\end{theorem}

It will be convenient to first prove an altered version from which our main theorem will follow:

\begin{theorem}
\label{gen-endpts}
Let $\F$ and $\G$ be $C^{1+\alpha}$-diffeomorphic IFSs defined on the same closed interval $I$
and suppose $\conv(F(I)) = \conv(G(I)) = I$ and $C(\F) \neq C(\G)$.
Let $\WW = (\W_n) \in \{\F,\G\}^\N$ and suppose there exists a $C^{1+\beta}$-diffeomorphic 
IFS $\h$ such that $C(\WW) = C(\h)$. 
Then $\WW$ is eventually periodic.
\end{theorem}

Note that in Theorem \ref{gen-endpts} the symmetry assumption and the bounds on the derivatives of functions in $\F$ and $\G$ have been removed, 
but we impose the additional condition that $\F$ and $\G$ preserve the endpoints
of $I$ in the sense that $\min C(\F) = \min C(\G) = \min I$ and $\max C(\F) = \max C(\G) = \max I$.
We will deduce Theorem \ref{gen-bdd-deriv} from Theorem \ref{gen-endpts} by showing that
the hypotheses of Theorem \ref{gen-bdd-deriv} imply this additional condition. 
(See Lemma \ref{preserve-endpoints}.)

\section{Proof of Theorem \ref{gen-endpts}}

The proof of Theorem \ref{gen-endpts} uses the fact that $C(\h)$
has approximately the same structure in basic intervals of arbitrarily small scale, while $C(\WW)$
will have different structures at these same scales.
To this end we first prove the following lemma.

\begin{lemma}
\label{diff-start}
Let $\F$ and $\G$ be injective IFSs defined on the same interval $I$
satisfying $\conv(F(I)) = \conv(G(I)) = I$,
and suppose that for two distinct words $\WW, \WW' \in \{\F, \G\}^\N$,
$C(\WW) = C(\WW')$. Then for any $\WW'', \WW''' \in \{\F,\G\}^\N$, we have $C(\WW'') = C(\WW''')$.
In particular, $C(\F) = C(\G)$.
\end{lemma}

\begin{proof}
Since $\WW$ and $\WW'$ are distinct, there is some minimal $k_0 \ge 0$ such that
$\W_{k_0+1} \neq \W_{k_0+1}'$.
But
\begin{eqnarray*}
C(\WW) = \bigcap_{n=0}^\infty (W_1 \circ \dots \circ W_n)(I) 
			&=& (W_1 \circ \dots \circ W_{k_0})\left( \bigcap_{n=k_0+1}^\infty W_{k_0+1} \circ \dots \circ W_n(I) \right)\\				&=& (W_1 \circ \dots \circ W_{k_0})\big(C(\W_{k_0+1}, \W_{k_0+2}, \dots)\big),
\end{eqnarray*}
so letting $K = C(\W_{k_0+1}, \W_{k_0+2}, \dots)$ and $K' = C(\W_{k_0+1}', \W_{k_0+2}', \dots)$, we have
$$(W_1 \circ \dots \circ W_{k_0})(K) = (W_1 \circ \dots \circ W_{k_0})(K').$$
Since the maps in $\F$ all have distinct images, as do the maps in $\G$, it follows that for each finite sequence
$(w_1, \dots, w_{k_0})$ with $w_i \in \W_i$, 
$$w_1 \circ \dots \circ w_{k_0}(K) = w_1 \circ \dots \circ w_{k_0}(K').$$
But $w_1 \circ \dots \circ w_{k_0}$ is a bijection onto its image, so we conclude $K = K'$.
Thus, we may assume without loss of generality that $k_0 = 0$ and
in particular that $\W_1 = \F$ and $\W_1' = \G$.

Let $\mathcal{A}$ be the collection of all intervals $A$ such that for some words $\WW'', \WW''' \in \{\F, \G\}^\N$,
$A$ is a gap of $C(\WW'')$ but not of $C(\WW''')$, and suppose, for a contradiction, that $\mathcal{A}$ is not empty.
Let $A$ be a gap of maximal length in $\mathcal{A}$.

If $A$ is a gap of $C(\WW'')$ for some $\WW''$ with $\W_1'' = F$, 
we consider two cases: the case that $A$ is a stage-$1$ gap,
i.e. a gap of $C_1(\WW'')$, and the case that it isn't. If $A$ is a stage-$1$ gap, then it
is a gap of $C_1(\F)$ and hence 
also a gap of every $\overline{\WW}$ with $\overline{\W}_1 = \F$. Otherwise, it must equal $f(A')$
for some $f \in \F$ and some gap $A'$ of $C(\W_2'', \W_3'', \dots)$. Since $f$ is a contraction,
by the maximality of $A$ we have that $A'$ is a gap of every $C(\overline{W})$. Hence, in this case,
as in the first case, we have that $A$ is a gap of every $C(\overline{\WW})$ with $\overline{\W}_1 = F$. 
The same argument shows that if $A$ is a gap of $C(\WW''')$ for some $\WW'''$ with $\W_1''' = \G$, 
then $A$ is a gap of every $C(\overline{\WW})$ with $\overline{\W}_1 = \G$.

Thus, if $A$ is a gap of some $C(\WW'')$ with $\W_1'' = \F$, then
it is a gap of every $C(\WW''')$ with $\W_1''' = \F$ and hence a gap of $C(\WW) = C(\WW')$. Since $\W_1' = \G$,
it follows that $A$ is also a gap of every $C(\WW''')$ with $\W_1''' = \G$, and therefore $A$ is a gap of
every $C(\WW''')$, contradicting $A \in \mathcal{A}$. It follows similarly that if $A$ is a gap of some
$C(\WW'')$ with $\W_1'' = G$ then $A$ is a gap of every $C(\WW''')$ again contradicting $A \in \mathcal{A}$.
Thus, no such $A$ can exist so the collection $\mathcal{A}$ must be empty.
\end{proof}

 Note that in the proof of Lemma \ref{diff-start} we use the fact that for any $k \in \N$, every gap of $C_k(\WW)$ 
 is a gap of $C(W)$.
 This does not generally hold if we remove the assumption that $\conv(F(I)) = \conv(G(I)) = \conv(I)$.

We also need the so-called Bounded Distortion Lemma (see \cite{Su}):
\begin{lemma}
\label{dist-lemma}
Let $I$ be a closed interval and
let $\F$ be any finite set of contracting $C^{1 + \alpha}$ diffeomorphisms $I \to I$.
Given a sequence $\mathbf{f} = (f_k) \in \F^\N$,
let $T^{(\mathbf{f}_k)}$ be the affine map sending $f_k \circ \dots \circ f_1 (I)$ onto $I$
such that $T^{(\mathbf{f}_k)} \circ f_k \circ \dots \circ f_1$ is orientation-preserving.
Then for any $0 < \alpha' < \alpha$ and any $(f_k)_{k=1}^\infty \in \F^\N$ the sequence $T^{(\mathbf{f}_k)} \circ f_k \circ \dots \circ f_1$
converges in the $C^{1+\alpha'}$ topology to a map 
$f \colon I \to I$. Furthermore, this convergence is uniform in the sequence $(f_k)$.
\end{lemma}

Before preceding to the proof of Theorem \ref{gen-endpts}
let us introduce some notation that will be useful in what follows.
Through the remainder of this section we assume that $\F$, $\G$, $\h$, and $I$ are as in the statement
of Theorem \ref{gen-endpts}.
Given $k \in \N$ and a finite sequence $\mathbf{w} = (w_1, \dots, w_k)$ of maps
with $w_i \in \W_i$, we call $B_\w = w_1 \circ \dots \circ w_k(I)$ a \emph{stage-$k$ basic interval}
of $C(\WW)$. If $T_\mathbf{w}$ is the affine map such that 
$$\Phi_{\w} \df T_{\mathbf{w}} \circ w_1 \circ \dots \circ w_k $$
is orientation-preserving and satisfies $\Phi_{\w}(I) = (I)$,
then we call $T_\w$ the \emph{renormalization map} for this basic interval.

It is instructive to consider a simpler case which still contains the main ideas of the general argument. 
Suppose every map in $\F$, $\G$, and $\h$ is affine, increasing, and defined on $[0,1]$, and that
$\conv(C(\h)) = \conv(C(\F)) = \conv(C(\G)) = [0,1]$.
Then for any $\w = (w_1, \dots, w_k)$,
$\Phi_\w$ is simply the identity map.
Furthermore, there exists $\gamma > 0$ such $x \in C(\h)$ if and only if $\gamma x \in C(\h)$.
In particular, for some open interval $A$
the sets $\gamma^i A$ are all gaps of $C(\h)$.
If $C(\h) = C(\WW)$, then these are
of course also gaps of $C(\WW)$. But there is a finite set of intervals $A_1, \dots, A_m$
such that every gap of $C(\WW)$ is of the form $w_1 \circ \dots \circ w_k(A_j)$
for some $w_i \in \W_i$ and $1 \le j \le m$. 
Thus, by the pigeonhole principle there are two gaps $\gamma^{i_1} A$ and 
$\gamma^{i_2}A = \gamma^{i_2 - i_1} (\gamma^{i_1} A)$
that are images of the same interval $A_j$ ($1 \le j \le m$).
In particular, we may find $1 \le j \le m$, $\ell  = i_2 - i_1 \in \N$, $\w = (w_1, \dots, w_{k_1})$,
and $\w' = (w_1', \dots, w_{k_2}')$ (with $w_i \in \W_i$ and $w_i' \in \W_i$) such that 
$$w_1' \circ \dots \circ w_{k_2}'(A_j) = \gamma^\ell w_1 \circ \dots \circ w_{k_1} (A_j).$$
Since the maps $w_i$ and $w_i'$ are all affine, it follows that 
$$B_{\w'} = w_1' \circ \dots \circ w_{k_2}'[0,1]= \gamma^\ell w_1 \circ \dots \circ w_{k_1}[0,1] = \gamma^\ell B_{\w}.$$
Moreover, by our choice of $\gamma$ this implies $B_{\w'} \cap C(\W) = \gamma^{\ell}(B_{\w} \cap C(\W))$.
Since $w_i$ and $w_i'$ are both increasing, we also obtain $T_{\w'} = \frac{1}{\gamma^\ell} T_\w$, so
we have both
$$T_{\w'}(B_{\w'} \cap C(\WW)) = T_{\w'}(w_1' \circ \dots \circ w_{k_2}'(C(\W_{k_2 + 1}, \W_{k_2 + 2}, \dots)))
	= C(\W_{k_2 + 1}, \W_{k_2 + 2}, \dots)$$
and
$$T_{\w'}(B_{\w'} \cap C(\WW)) = T_\w(B_{\w} \cap C(\WW)) = C(\W_{k_1 + 1}, \W_{k_1 + 2}, \dots)$$
Hence, $C(\W_{k_2 + 1}, \W_{k_2 + 2}, \dots) = C(\W_{k_1 + 1}, \W_{k_1 + 2}, \dots)$.
But if $\WW$ is not eventually periodic then these two tails must be different, which contradicts
Lemma \ref{diff-start}.

The general argument is more complicated but the overall structure is similar.
The main difference is that the functions $\Phi_\w$ will no longer be equal to the identity function.
We will instead rely on an application of the Bounded Distortion Lemma, which forces us
to choose the tuples $\w$ and $\w'$ more carefully, so that the corresponding maps $\Phi_\w$
and $\Phi_{\w'}$ are close. 

\begin{proof}[Proof of Theorem \ref{gen-endpts}]
Translating and rescaling we may assume $I = [0,1]$.
Suppose $\WW$ is not eventually periodic and let $\h$ be a $C^{1+\beta}$-diffeomorphic 
IFS defined on an interval
$I'$ such that
$C(\WW) = C(\h)$.
We may assume without loss of generality that $\conv(H(I')) = I'$ so that in particular 
$I' = [0,1]$.
Indeed, if not then since $\min C(\h) = \min C(\WW) = 0$ and $\max C(\h) = \max C(\WW) = 1$
we may let $\overline{h} = h|_{[0,1]}$ and $\overline{\h} = \{\overline{h} \colon h \in \h\}$,
which satisfies $C(\h) = C(\overline{\h})$.
Let $h_1, h_2 \in \h$ be the maps whose images contain $0$ and $1$ respectively.
Note that we can also assume that 
either $h_1(0) = 0$ or $h_2(1) = 1$ since otherwise
we can consider $\h^2 = \{h \circ h' \colon h, h' \in \h\}$, which satisfies $C(\h^2) = C(\h)$
and contains $h_1 \circ h_2$ with $h_1(h_2(0)) = 0$.
Thus, reflecting both $C(W)$ and $C(\h)$ if necessary, 
we may assume that for any $x \in [0,1]$, 
	\begin{equation}
	\label{scale-down-points}
	x \in C(\h) \Leftrightarrow h_1(x) \in C(\h)
	\end{equation}
	for a map $h_1 \in \h$ satisfying $h_1(0) = 0$. Let $\gamma = h_1'(0) > 0$.

Let $\text{Gap}_0$ be the set of all gaps of $C_1(\F)$ and $C_1(\G)$.	
Fix a gap $I_0$ of $C_1(\h)$ and note that for each $k \in \N$, $h_1^k(I_0)$ is a gap of $C(\h) = C(\WW)$.
Hence, for each $m \in \N$, there exist $k_m$, $A_m \in \gap_0$, and $w^{(m)}_i \in \W_i$ ($1\le i \le k_m$)
such that
$$h_1^m(I_0) = w^{(m)}_1 \circ \dots \circ w^{(m)}_{k_m} (A_m)$$
Let $B_m = w^{(m)}_1 \circ \dots \circ w^{(m)}_{k(m)} [0,1]$,
i.e. $B_m$ is the basic interval of $C(\WW)$ that contains $h_1^m(I_0)$.
Note that by construction these are nested.
By the pigeonhole
	principle there exist infinitely many distinct pairs $i_1, i_2 \in \N$
	with $0 < i_2 - i_1 \le \# \text{Gap}_0 \times [\#(\F \cup \G)]^k$
	such that
	$A_{i_1} = A_{i_2}$ and $w_{k_{i_1} - i}^{(i_1)} = w_{k_{i_2} - i}^{(i_2)}$
	for $0 \le i \le k$ (and in particular $\W_{k_{i_1}} = \W_{k_{i_2}}$).
	Again using the pigeonhole principle we may assume that $T_{\w^{(i_1)}}$ and $T_{\w^{(i_2)}}$
	have the same orientation.
	Let us say that such a pair is \emph{$k$-compatible}.
	The strategy of the proof is to find distinct ``suffix'' sequences $\mathbf{S}^{(1)}$ and
	$\mathbf{S}^{(2)}$ in $\{\F, \G\}^\N$ and a ``prefix'' sequence of maps $\varphi_i \in \F \cup \G$
	such that for each $k, N \in \N$ there exists a $k$-compatible pair $(i_1, i_2)$
	with $i_1 \ge N$ such that for each $0 \le j \le k$ and $m \in \{1, 2\}$ we have
	$w^{(i_m)}_{k_{i_m} - j} = \varphi_j$ and $\W_{k_{i_m} + j} = \mathbf{S}^{(m)}_j$.
	Before proving the existence of such sequences, let us first show that this suffices to complete the proof.

	Since $h_1$ is $C^{1}$ and $i_2 - i_1 \le \# \text{Gap}_0 \times [\#(\F \cup \G)]^k$,
	for any $\vre > 0$ we have, for large enough $i_1$ (relative to $k$),
	$$|h_1^{i_2 - i_1}(x) - \gamma^{i_2 - i_1} x| < \vre |B_{i_2}|$$ for all $x \in B_{i_1}$.
	Indeed, there are finitely many such functions $h^{i_2 - i_1}$, given the bound on $i_2 - i_1$,
	and for each one we have that $\dfrac{d}{dx} \bigg\rvert_{x = 0} h^{i_2 - i_1} = \gamma^{i_2 - i_1}$.
	But for any function $g$ with $g(0) = 0$ and $g'(0) = D$ and any $\epsilon' > 0$ we have
	$|g(x) - Dx| < \epsilon' x$ for small enough $x$,
	by the definition of the derivative.
	Since $x \in B_{i_1}$ and the ratio $|B_{i_2}|/|B_{i_1}|$ is uniformly bounded due to the 
	bound on $i_2 - i_1$, the stated inequality follows.
	This implies that the Hausdorff distance between 
	$B_{i_2} \cap C(\h)$ and $\gamma^{i_2 - i_1} (B_{i_1} \cap C(\h))$ is at most $\epsilon |B_{i_2}|$,	
	and hence the Hausdorff distance between the rescalings
	$T_{\w^{(i_1)}} (B_{i_1} \cap C(\WW))$
	and $T_{\w^{(i_2)}}(B_{i_2} \cap C(\WW))$ is at most $\vre$.
	But these sets are $\Phi_{\w^{(i_1)}}(C(\W_{k_{i_1} + 1}, \dots))$
	and $\Phi_{\w^{(i_2)}}(C(\W_{k_{i_2} + 1}, \dots))$ respectively.
	By Lemma \ref{dist-lemma}, there exists a diffeomorphism $\Phi$ such that
	$\Phi_{\w^{(i_1)}}$ and $\Phi_{\w^{(i_2)}}$ are $\varepsilon$-close
	to $\Phi$ for large enough $k$.
	This implies that $\Phi(C(\W_{k_{i_1} + 1}, \dots))$ and $\Phi(C(\W_{k_{i_2} + 1}, \dots))$
	are $\varepsilon$-close for large $k$ and $i_1$.
	Since $\Phi$ is a diffeomorphism and $k$ and $i_1$ can be taken arbitrarily large by assumption,
	it follows that $C(\mathbf{S}^{(1)}) = C(\mathbf{S}^{(2)})$,
	which contradicts Lemma \ref{diff-start}.

	We now prove the existence of the desired `prefix' and `suffix' sequences.
	Note that we have already seen that there exist infinitely many $k$-compatible pairs for each $k \in \N$.
	The only remaining obstacle to extracting the desired sequences is the possibility
	that the tails $(W_{k_{i_1}}, \dots)$ and $(W_{k_{i_2}}, \dots)$ share very long initial strings,
	with length growing as $k$ does, so that the sequences $\mathbf{S}^{(1)}$ and $\mathbf{S}^{(2)}$
	would be equal. To prove that this is not the case, first
	observe that there exists $J \in \N$, depending only on $\F$, $\G$, and $\h$
	such that if $A$ is a gap of $C_n(\WW')$ for some $\WW' \in \{\F, \G\}^\N$ and $n \in \N$,
	then $h(A)$ is a gap of $C_{n+J}(\WW')$ for each $h \in \h$. 
	Suppose $\W_{k_{i_1} + j} = \W_{k_{i_2} + j}$
	for $1 \le j \le J$ and consider $h^{i_1 + 1}(I_0)$ and $h^{i_2 + 1}(I_0)$.
	There exist $w_{k_{i_1} + j}$ for $1\le j \le j_1$ (where $1 \le j_1 \le J$) and $A_{i_1}' \in \text{Gap}_0$
	such that $w_1 \circ \dots \circ w_{k_{i_1} + j_1}(A_{i_1}') = h_1^{i_1 + 1}(I_0)$.
	Similarly, there exist $w_{k_{i_2} + j}$ for $1\le j \le j_2$ (where $1 \le j_2 \le J$) and $A_{i_2}'$
	with $w_1 \circ \dots \circ w_{k_{i_2} + j_2}(A_{i_2}') = h_1^{i_2 + 1}(I_0)$.
	But for large enough $i_1$, 
	the Hausdorff distance between $h_1^{i_2 + 1}(I_0)$ and $\gamma^{i_2 - i_1} h_1^{i_1 + 1}(I_0)$
	is at most $\vre|B_{i_2}|$
	and hence the Hausdorff distance between
	$\Phi_{\w^{(i_1)}}(w_{k_{i_1}} \circ \dots \circ w_{k_{i_1} + j_1}(A_{i_1}'))$ and 
			$\Phi_{\w^{(i_2)}}(w_{k_{i_2}} \circ \dots \circ w_{k_{i_2} + j_2}(A_{i_2}'))$
	is at most $\vre$.		
	But by Lemma \ref{dist-lemma} and the fact that the diffeomorphisms in $\F$ have distinct images,
	as do the diffeomorphisms in $\G$, this implies that for sufficiently large $i_1$, we have
	$j_1 = j_2$, $A_{i_1}' = A_{i_2}'$ and
	$w^{(i_1)}_{k_{i_1} + j} = w^{(i_2)}_{k_{i_2} + j}$ for $1 \le j \le j_1$.
	Thus, $(i_1 + 1, i_2 + 1)$ is $k$-compatible. Continuing in this way, and using the fact that $\WW$ is
	not eventually periodic, we may assume that $\W_{k_{i_1} + j} \neq \W_{k_{i_2} + j}$ for some $1 \le j \le J$.
	Call such a pair $(k, J)$-compatible.
	
	Consider the sets $\Omega_k = ((\F \cup \G) \times \{\F, \G\}^{2})^k$. 
	We fix some order on the finite set $(\F \cup \G) \times \{\F, \G\}^{2}$ and put
	the associated lexicographic ordering on $\Omega_k$.
	For each $k$, let 
	$$\omega_k = (\varphi_0, \s^{(1)}_1, \s^{(2)}_1, \dots, \varphi_{k-1}, \s^{(1)}_k, \s^{(2)}_k)$$
	be the minimal element of $\Omega_k$ such that for infinitely many $(k,J)$-compatible
	pairs $(i_1, i_2)$, we have
	$$ w^{(i_1)}_{k_{i_1} - i} = w^{(i_2)}_{k_{i_2} - i} = \varphi_i \text{ for } 0 \le i < k$$
	and, for $m \in \{1,2\}$
	$$ \W_{k_m + j} = \s^{(m)}_j \text{ for } 1\le j \le k$$
	Notice that for any $k > J$ there exists $1 \le j \le J$ such that $\s^{(1)}_j \neq \s^{(2)}_j$.
	By construction each $\omega_{k}$ extends the previous one so we obtain
	infinite sequences $(\varphi_k)$, $\mathbf{S}^{(1)}$, and $\mathbf{S}^{(2)}$ with the desired properties.
	\end{proof}

\section{Proof of Theorem \ref{gen-bdd-deriv}}
Theorem \ref{gen-bdd-deriv} follows from Theorem \ref{gen-endpts} once we reduce to the case that
$\conv(F(I)) = \conv(G(I)) = I$, so it suffices to prove the following lemma.
\begin{lemma}
\label{preserve-endpoints}
Under the assumptions of Theorem \ref{gen-bdd-deriv},
either $\WW$ is eventually periodic or
there exist $C^{1+\alpha}$-diffeomorphic IFSs $\overline{\F}$ and $\overline{\G}$
defined on the same closed interval $I$ such that
\begin{itemize}
	\item $\conv(\overline{F}(I)) = \conv(\overline{G}(I)) = I$. 
	\item For any $\WW' \in \{\F, \G\}^\N$, $C(\WW') = C(\overline{\WW'})$, where
			$\overline{\W_i'} = \overline{\F}$ if $\W_i' = \F$ and $\overline{\W_i'} = \overline{\G}$ if $\W_i = \G$.
\end{itemize}
\end{lemma}

To prove this lemma we will compare the relative positions of gaps of $C(\WW)$ to those of $C(\h)$.
Given a compact set $K$, we define the \emph{relative position} of a gap $A = (a,b)$ of $K$
to be $\pos_K(A) = \frac{a - \min K}{b-a}$.
We set
\[
\pos(K, \delta) = \{\pos_K(a,b) \colon (a,b) \text{ is a gap of } K \text{ with } |b-a| < \delta\}.
\] 

\begin{proof}
Without loss of generality, assume $I = [0,1]$ and $\min C(\G) \le \min C(\F)$.
For each $f\in \F$ and $g \in \G$, let $g^\ast = g|_{[\min C(\G), 1]}$ and $f^\ast = f|_{[\min C(\G), 1]}$.
Set $\G^\ast = \{g^\ast \colon g \in \G\}$ and $\F^\ast= \{f^\ast \colon f \in \F\}$,
which are well-defined IFSs on $[\min C(\G), 1]$.
Since 
$$\lim_{n \to \infty} w_1 \circ \dots \circ w_{i_n}(x) = \lim_{n \to \infty} w_1 \circ \dots \circ w_{i_n}(1)$$
for any sequence of functions $w_i \in \F \cup \G$ and any $x \in [0,1]$,
it follows that
for any $\WW' \in \{\F, \G\}^\N$, we have $C(\WW') = C(\WW^\ast)$, where $\W_i^\ast = \F^\ast$ if $\W_i' = \F$ and $\W_i^\ast = \G^\ast$ if $\W_i' = \G$.
If $\min C(\F^\ast) = \min C(\G^\ast)$ then by our symmetry assumption we have produced the desired IFSs.
Thus, after translating and rescaling we may assume 
(for a contradiction) that
$I = [0,1]$, $\min C(\G) = 0$, and $\min C(\F) = \vre \in (0,1/2)$.
Let $g_1 \in \G$ with $g_1(0) = 0$ and $f_1 \in \F$ with $f_1(0) \le \vre$.
Note that since $g_1'(x) \in (0,1/2)$ for all $x \in [0,1]$, we have
\begin{equation}
\label{g-1-bound}
g_1(x) < \frac12 x\text{ for all } x \in (0,1].
\end{equation}
Moreover, since $\vre = \min C(\F)$, we must have $f_1(\vre) = \vre$, so
since $f_1'(x) \in (0,1/2)$ for all $x\in [0,1]$, it follows
that 
\begin{equation}
\label{f-1-bound}
f_1(x) \ge \vre  - \frac12(\vre - x) = \frac12\vre + \frac12 x\text{ for all }x \in [0,\vre].
\end{equation}
We may also assume without loss of generality, as before,
that there exist $h_1, h_2 \in \h$ that fix $\min C(\h)$ and $\max C(\h)$ respectively.
We wish to apply Lemma \ref{dist-lemma} to conclude that
for any gap $A$ of $C(\h)$ the sequence $\pos_{C(\h)}(h_1^k(A))$ is Cauchy.
To that end, we first show that the restriction of $h_1$ to some neighborhood of $\min C(\h)$ is a contraction. 
Notice that, since $h_1'$ is continuous, there exists $t > \min C(\h)$ such that either
$h_1'(x) \le 1$ for all $x \in [\min C(\h),t]$ or $h_1'(x) \ge 1$ for all $x\in [\min C(\h), t]$.
In the former case, $h_1|_{[\min C(\h),t]}$ is
a contraction, so assume for a contradiction that the latter holds.
This implies $h_1[\min C(\h),t] \supset [\min C(\h), t]$, and hence 
$h_1^k[\min C(\h),1] \supset [\min C(\h),t]$ for all $k \in \N$.
It follows that $[\min C(\h),t] \subset C(\h)$. But since we have assumed that $f'(x) \in (0,1/2)$ for
every $f \in \F \cup \G$ and every $x \in I$ this clearly implies $C(\h) \neq C(\WW)$, a contradiction.
Now, if there does not exist $k \in \N$ such that $h_1^k[\min C(\h),1] \subset [\min C(\h),t]$, 
then again we have that
$[\min C(\h),t] \subset C(\h)$, yielding a contradiction, so for some large enough $k_0$ we have that 
$h_1^{k_0}(A) \subset [\min C(\h), t]$ and hence the sequence of sets 
$\pos_{C(\h)}(h_1^k(A)) = \pos_{C(\h)}(h_1^{k-k_0}(h_1^{k_0}(A)))$
is Cauchy by Lemma \ref{dist-lemma}.

Now if $M > 0$
then there is some $n \in \N$ such that any gap of $C(\h)$ 
with relative position at most $M$ must be of the form $h_1^\ell \circ h_{i_1} \circ \dots \circ h_{i_n}(a,b)$,
where $(a,b)$ is a gap of $C_1(\h)$ and $h_{i_j} \in \h$ for $1\le j \le n$.
Hence, since we have shown that $\pos_{C(\h)}(h_1^\ell(A))$ is Cauchy for any gap $A$, 
there exist $p_1, \dots, p_m > 0$ such that for any $\eta > 0$
there exists $\delta > 0$ such that 
\begin{equation}
\label{rel-pos}
\pos(C(\h), \delta) \cap (0,M) \subset \{p_1, \dots, p_m\}^{(\eta)},
\end{equation}
where $Z^{(\eta)}$ denotes the $\eta$-neighborhood of a set $Z$.

We are going to arrive at a contradiction by producing gaps of $C(\WW)$ that cannot all be gaps of
$C(\h)$ due to their distances from each other, combined wth (\ref{rel-pos}).
To do this we first show that if $\WW', \WW'' \in \{\F, \G\}^\N$ have a long initial string in common
then $C(\WW')$ and $C(\WW'')$ contain gaps with very close, but not equal, relative positions.
In particular, we prove the following claim.
\begin{claim}
\label{init-str-claim}
Let $\WW', \WW'' \in \{\F, \G\}^\N$ with $\W_1' = \W_1''$ and $\min C(\WW') > \min C(\WW'')$.
Let $A$ be a gap of $C_1(\WW') = C_1(\WW'')$ and let $A' = (a',b')$ and $A'' = (a'', b'')$ 
be the gaps of $C(\WW')$ and $C(\WW'')$ 
respectively that contain $A$. 
Then there exist $\beta_k, \gamma_k > 0$ depending only on $\F$ and $\G$ such that if
$k_0 = \min\{ k \in \N \colon \W_k' \neq \W_k''\}$, then 
\begin{eqnarray*}
& \beta_{k_0} < \min C(\WW') - \min C(\WW'') < \gamma_{k_0}\\
&\beta_{k_0} < a'' - a' < \gamma_{k_0}\\
&\beta_{k_0} < b' - b'' < \gamma_{k_0}\\
&\beta_{k_0} < \pos_{C(\WW'')}(A'') - \pos_{C(\WW')}(A') < \gamma_{k_0}
\end{eqnarray*}
\end{claim}

\begin{subproof}[Proof of Claim \ref{init-str-claim}]
First note that if $\W_{k_0}' = \F$ and $\W_{k_0}'' = \G$, then
letting $\varphi = w_1 \circ \dots \circ w_{k_0 - 1}$, where $w_i = f_1$ if $\W_i' = \F$
and $w_i = g_1$ if $\W_i' = \G$, we have
$\min C(\WW') \ge \varphi(f_1(0))$ and
 $\min C(\WW'') \le \varphi(g_1(\vre))$.
 But by (\ref{g-1-bound}) and (\ref{f-1-bound}), $f_1(0) \ge \frac12 \vre$ and $g_1(\vre) < \frac12 \vre$.
 Let $\beta_0 = f_1(0) - g_1(\vre) > 0$.
 Now, since $\F \cup \G$ is a finite set of diffeomorpisms defined on a compact set
 and the ranges of their derivatives lie in $(0,1/2)$,
 there exists $0 < \lambda < 1/2$ such that
 for any $w \in \F \cup \G$ and any $x, y \in [0,1]$,
 $$\lambda |y-x| < |w(y)-w(x)| < \frac{1}{2} |y-x|.$$ 
But then  
\begin{equation}
\label{double-bound-mins}
1/2^{k_0} > \varphi(1) - \varphi(0) \ge \min C(\WW') - \min C(\WW'') \ge \varphi(f_1(0)) - \varphi(g_1(\vre)) > \beta_0 \lambda^{k_0} 
\end{equation}
The same argument shows that $\W_{k_0}' = \G$ is impossible by assumption.
Thus, the first of our desired inequalities holds as long as $\beta_k \le \beta_0 \lambda^k$
and $\gamma_k \ge 2^{-k}$.

Now let $L' = \min C(\W_2', \W_3', \dots)$ and $L'' = \min C(\W_2'', \W_3'', \dots)$.
Also let $w_L, w_R \in \F \cup \G$ such that the gap 
$A = (a,b)$ of $C_1(\WW')$ sits between $w_L[0,1]$ and $w_R[0,1]$,
i.e. $a = w_L(1)$ and $b = w_R(0)$. We have $b' = w_R(L')$ and $b'' = w_R(L'')$.
Furthermore, since $\F$ and $\G$ are symmetric, it follows that
$\max C(\W_2', \W_3', \dots) = 1- L'$ and $\max C(\W_2'', \W_3'', \dots) = 1- L''$,
so $a' = w_L(1-L')$ and $a'' = w_L(1- L'')$.
Thus by our choice of $\lambda$ we have
\begin{eqnarray*}
	a'' - a' &=& w_L(1-L'') - w_L(1- L') < \frac{1}{2}\left((1-L'') - (1-L')\right) =  \frac{1}{2} (L' - L'')\\				
	&<& \frac{1}{2\lambda} (w_1(L') - w_1(L''))
		= \frac{1}{2\lambda} \left(\min C(\WW') - \min C(\WW'') \right)
\end{eqnarray*}
Arguing similarly, we obtain $a'' - a' > 2\lambda \left(\min C(\WW') - \min C(\WW'') \right)$, along
with identical bounds on $b' - b''$, so
our first three desired inequalities hold as long as
$\beta_k \le 2 \beta_0 \lambda^{k+1}$ and $\gamma_k \ge \dfrac{1}{2^{k+1}\lambda}$.

In particular, we have $a' < a''$ and $b' > b''$, so
\begin{eqnarray*}
\pos_{C(\WW'')}(A'') - \pos_{C(\WW')}(A')
&= \dfrac{a'' - \min C(\WW'')}{b'' - a''} -  \dfrac{a' - \min C(\WW')}{b' - a'}  \\
& > \min C(\WW') - \min C(\WW'').
\end{eqnarray*}
We wish to obtain a bound in the other direction as well, i.e. we wish to find $\gamma_0$
 such that
 \begin{equation}
 \label{initial-bound}  
 \dfrac{a'' - \min C(\WW'')}{b'' - a''} -  \dfrac{a' - \min C(\WW')}{b' - a'} 
<  \gamma_0 \left(\min C(\WW') - \min C(\WW'')\right).
\end{equation}
Note that for any $n_1, n_2, d_1, d_2 > 0$, we have
$$\dfrac{n_1}{d_1} - \dfrac{n_2}{d_2}= \dfrac{n_1 d_2 - n_2 d_1}{d_1d_2}
= \dfrac{n_1(d_2 - d_1) - d_1(n_2 - n_1)}{d_1d_2}
\le \dfrac{n_1|d_2 - d_1| + d_1|n_2 - n_1|}{d_1 d_2}.$$
Thus, since our numerators and denominators are all bounded above by $1$
and the denominators $b' - a'$ and $b'' - a''$ are bounded below by the smaller of the
minimal gap lengths of $F[0,1]$ and $G[0,1]$,
it will suffice to bound the difference between the numerators and the difference between the denominators.
Using our bounds on $a'' - a'$ and $b' - b''$ we have
$$|(a' - \min C(\WW')) - (a'' - \min C(\WW'')| < \left(1 + \frac{1}{2\lambda}\right) |\min C(\WW') - \min C(\WW'')|$$
and 
$$|(b' - a') - (b'' - a'')| < \frac{1}{\lambda} |\min C(\WW') - \min C(\WW'')|.$$
Thus, taking $\gamma_0 = \dfrac{2}{\lambda (b_0 - a_0)^2}$, where $(a_0,b_0)$ is 
of minimal length among all gaps of $F[0,1]$ and $G[0,1]$, we obtain (\ref{initial-bound}).
Letting $\beta_k = 2\beta_0\lambda^{k+1}$ and $\gamma_k = \gamma_0 /2^k$ we obtain 
all four desired inequalities.
\end{subproof}

We will use this claim to produce the desired gaps in $C(\WW)$ 
by finding long strings in $\WW$ which agree on exactly the right number of initial terms.
Say that a string $(\s_1, \dots, \s_k) \in \{\F, \G\}^k$ is \emph{$\WW$-ambiguous} if for any $N \in \N$, 
there exist $\ell, m \ge N$ with
$$(\W_{\ell}, \dots, \W_{\ell + k}) = (\s_1, \dots, \s_k, \F)$$
and
$$(\W_{m }, \dots, \W_{m + k}) = (\s_1, \dots, \s_k, \G).$$
Note that there are $\WW$-ambiguous strings of any length.
Indeed, suppose for a contradiction that for any string $\mathbf{S}$ of length $k$,
all occurrences of this string appearing late enough in $\WW$ are followed by the same symbol.
That is, suppose there exist
$N_\mathbf{S} \in \N$ and $\s_\mathbf{S} \in \{\F, \G\}$
such that for any $\ell \ge N_\mathbf{S}$
with $(\W_\ell, \dots, \W_{\ell + k - 1}) = \mathbf{S}$, we have $\W_{\ell + k} = \s_\mathbf{S}$.
Then letting 
$N = \max N_\mathbf{S}$ over all $\mathbf{S} \in \{\F, \G\}^k$ and setting
$\overline{\WW} = (\W_N, \W_{N+1}, \dots)$, we see that for any $k' \ge k$ we have
$P_{\overline{\WW}}(k') = P_{\overline{\WW}}(k)$.
(Here, $P_{\WW}$ is the complexity function $P_{\WW}(k) = \#\{(\s_1, \dots, \s_k) \colon (\W_{n+1}, \dots, \W_{n+k}) = (\s_1, \dots, \s_k)\text{ for some }n \in \N\}$.)
By the Morse-Hedlund Theorem (see \cite{MH}), any sequence with a bounded complexity function is
eventually periodic, and hence $\WW$ is eventually periodic,
contradicting our assumptions.

Let $\mathcal{A}_k$ be the set of $\WW$-ambiguous strings of length $k$.
Let $\s_1 \in \{\F, \G\}$ have the property that for infinitely many $k$, 
there exists $\mathbf{S} = (\s_k, \dots, \s_1) \in \mathcal{A}_k$.
Note that any tail of a $\WW$-ambiguous string is also $\WW$-ambiguous, so it follows
that for every $k \in \N$, there exists $\mathbf{S} = (\s_k, \dots, \s_1) \in \mathcal{A}_k$.
Inductively, if $\s_1, \dots, \s_m$ are chosen such that all $k \in \N$
there exists $\mathbf{s} = (\s_k, \dots, \s_m, \dots, \s_1) \in \mathcal{A}_k$,
choose $\s_{m+1}$ so that for all $k \in \N$,
there exists $\mathbf{s} = (\s_k, \dots, \s_{m+1}, \dots, \s_1) \in \mathcal{A}_k$.
We thus obtain 
a sequence $(\mathcal{S}_j)_{j \in \N}$
such that for each $k \in \N$, $(\mathcal{S}_k, \dots, \mathcal{S}_1)$ is a $\W$-ambiguous string.
\ignore{
Let $\varphi_i = f_1$ if $\W_i = \F$ and $\varphi_i = g_1$ if $\W_i = \G$,
and let 
$$\Phi_k = T^{(\varphi_1, \dots, \varphi_k)} \circ \varphi_1 \circ \dots \circ \varphi_k.$$
By Lemma \ref{dist-lemma}, if $(k_j)$ is a subsequence such that for any $\ell$
$$(\W_{k_j - \ell}, \dots, \W_{k_j}) = (\s_{\ell+1}, \dots, \s_1)$$
for all sufficiently large $j$, then
$\Phi_{k_j}$ converges to a diffeomorphism $\Phi$ of $[0,1]$.} 

To arrive at a contradiction to (\ref{rel-pos}), we first choose the constants $M$ and $\delta$.
The sequence space $\{f_1, g_1\}^\N$ with the usual topology is compact.
Furthermore, by Lemma \ref{dist-lemma}, for each $\w = (w_1, w_2, \dots)$ in this space there is a
diffeomorphism $\Phi_{\w} = \lim_{k \to \infty} \Phi_{(w_k, \dots, w_1)}$ and the mapping
$\w \mapsto \min_{x\in[0,1]} |\Phi_{\w}'(x)|$ is continuous, as is the corresponding function with $\min$
replaced by $\max$.
Thus, there exist $0 < \lambda_1 < \lambda_2$ such that
for any $\w \in \{f_1, g_1\}^\N$ and any $x \in [0,1]$, $\Phi_\w'(x) \in (\lambda_1, \lambda_2)$.
Fix $M$ large enough that there exist gaps of $C_1(\F)$ and $C_1(\G)$ with relative positions less than 
$\lambda_1 M/\lambda_2$.
Let $p_1, \dots, p_m$ be as in (\ref{rel-pos}) and 
choose $k$ large enough that $2\gamma_k < |p_i - p_j|$ for every $i \neq j$.
Choose $\delta$ small enough that (\ref{rel-pos}) holds with $\eta = \gamma_k/3$.
We are going to find limit points of $\pos(C(\WW), \delta) \cap (0,M)$
with distance less than $\gamma_k$, contradicting (\ref{rel-pos}).
For each $\ell \in \N$ we can find $i_1^{(\ell)}, i_2^{(\ell)} \in \N$ such that
$$(\W_{i_1^{(\ell)} - \ell}, \dots, \W_{i_1^{(\ell)} + k}) = (\mathcal{S}_{k + \ell}, \dots, \mathcal{S}_1, \F)$$
and
$$(\W_{i_2^{(\ell)} - \ell}, \dots, \W_{i_2^{(\ell)} + k}) = (\mathcal{S}_{k + \ell}, \dots, \mathcal{S}_1, \G).$$
Let $\w_m= (w_1, \dots, w_m)$
where $w_i = f_1$ if $\mathcal{W}_i = \F$ and $w_i = g_1$ if $\mathcal{W}_i = \G$,
and let $\Phi^{(\ell)}_1 = \Phi_{\w_{i^{(\ell)}_1-1}}$ and
$\Phi^{(\ell)}_2 = \Phi_{\w_{i^{(\ell)}_2-1}}$.
Now let $A^{(\ell)}$ be a gap of $C_1(\W_{i^{(\ell)}_1})$ with relative position less than 
$\lambda_1 M/\lambda_2$
and let $A_1^{(\ell)}$ and $A_2^{(\ell)}$ be the corresponding gaps of $C(\W_{i_1^{(\ell)}}, \dots)$ and $C(\W_{i_2^{(\ell)}}, \dots)$.
Then their relative positions are even smaller than that of $A^{(\ell)}$,
and by 
Claim \ref{init-str-claim} they satisfy
$\beta_k < |\pos_{C(\W_{i_1^{(\ell)}}, \dots)}(A_1^{(\ell)}) - \pos_{C(\W_{i_2^{(\ell)}}, \dots)}(A_2^{(\ell)})| < \gamma_k$.

Using the pigeonhole principle and moving to a subsequence we may also assume for any $j \in \N$ that 
for some fixed suffix sequences $\mathbf{S}_1$ and $\mathbf{S}_2$ we have
$(\W_{i_1^{(\ell)}+k + 1}, \dots, \W_{i_1^{(\ell)}+k + j}) = \mathbf{S}_1$
and
$(\W_{i_2^{(\ell)}+k + 1}, \dots, \W_{i_2^{(\ell)}+ k + j}) = \mathbf{S}_2$
for all sufficiently large $\ell \in \N$.
In particular, again applying Claim \ref{init-str-claim}
we may assume without loss of generality that the diameter of the set $\{A_1^{(\ell)} : \ell \ge \ell_0\}$
(with respect to the Hausdorff metric) is less than $1/\ell_0$, and similarly for $A^{(\ell)}_2$.
Therefore, setting $(a_1^{(\ell)}, b_1^{(\ell)}) = A_1^{(\ell)}$ and $(a_2^{(\ell)}, b_2^{(\ell)}) = A^{(\ell)}_2$,
we can let $a_1 = \lim a_1^{(\ell)}$, $a_2 = \lim a_2^{(\ell)}$, $b_1 = \lim b_1^{(\ell)}$, and $b_2 = \lim b_2^{(\ell)}$.
Similarly, we can assume without loss of generality that the limits
$c_1 = \lim \min C(\W_{i_1^{(\ell)}}, \dots)$ and $c_2 = \lim \min C(\W_{i_2^{(\ell)}}, \dots)$ exist.

Note that, applying Claim \ref{init-str-claim} once more,
we have that $a_1 - a_2$, $b_2 - b_1$, and $c_2 - c_1$ are either all positive or all negative.
Since $\W_{i^{(\ell)}_1 - j} = \W_{i^{(\ell)}_2 - j}$ for $1 \le j \le \ell$,
we have $\Phi \df \lim \Phi^{(\ell)}_1 = \lim \Phi^{(\ell)}_2$.
Since $\Phi$ is strictly increasing, we also have that
$\Phi(a_1) - \Phi(a_2)$, $\Phi(b_2) - \Phi(b_1)$, and $\Phi(c_2) - \Phi(c_1)$ are either all positive
or all negative and therefore
$$\left|\dfrac{\Phi(a_1) - \Phi(c_1)}{\Phi(b_1) - \Phi(a_1)} 
			- \dfrac{\Phi(a_2) - \Phi(c_2)}{\Phi(b_2) - \Phi(a_2)} \right| > 0.$$
But $\dfrac{\Phi_1^{(\ell)}(a_1^{(\ell)}) - \Phi_1^{(\ell)}(c_1^{(\ell)})}
							{\Phi_1^{(\ell)}(b_1^{(\ell)}) - \Phi_1^{(\ell)}(a_1^{(\ell)})}$
and
$\dfrac{\Phi_2^{(\ell)}(a_2^{(\ell)}) - \Phi_2^{(\ell)}(c_2^{(\ell)})}
							{\Phi_2^{(\ell)}(b_2^{(\ell)}) - \Phi_2^{(\ell)}(a_2^{(\ell)})}$
are relative positions of $C(\WW)$ 
which tend to $\dfrac{\Phi(a_1) - \Phi(c_1)}{\Phi(b_1) - \Phi(a_1)} $ and 
$\dfrac{\Phi(a_2) - \Phi(c_2)}{\Phi(b_2) - \Phi(a_2)} $ respectively.
Furthermore, since $i_1^{(\ell)}$ and $i_2^{(\ell)}$ clearly must tend to infinity,
the gaps $w_1 \circ \dots \circ w_{i^{(\ell)}_1 - 1}(A^{(\ell)}_1)$ and 
$w_1 \circ \dots \circ w_{i^{(\ell)}_2 - 1}(A^{(\ell)}_2)$ have length less than $\delta$
for sufficiently large $\ell$.
Lastly, by our choice of $\lambda_1$ and $\lambda_2$,
the relative positions of these gaps in $C(\W)$ are less than 
$\dfrac{\lambda_2}{\lambda_1}\left( \dfrac{\lambda_1 M}{\lambda_2} \right) = M$.
Thus, $\pos(C(\WW), \delta) \cap (0, M)$ contains two limit points with distance less than $\gamma_k$,
contradicting (\ref{rel-pos}).
\end{proof}

It is sometimes possible to determine that a set is not the limit set of an IFS by showing that its box-counting dimension differs from
its Hausdorff dimension. The next example shows that our methods above cover some cases where this dimension argument is not possible.

\begin{exa}
\label{dims-equal}
Let $0 < \rho, \vre < 1$ and $k \in \N$.
Let $\F$ and $\G$ be homogeneous IFSs with contraction ratio $\rho$ defined on an interval $I$
such that $F(I) = \bigcup_{n=1}^k I_n$ and $G(I) = \bigcup_{n=1}^k J_n$,
where $I_n$ and $J_n$ are closed intervals of length $\rho$.
Suppose that all of these intervals are $\vre$-separated, i.e. suppose that
for any $1 \le n \le m \le k$, we have $\dist(I_n, J_m) \ge \vre$
and that if $n < m$, we also have $\dist(I_n, I_m) \ge \vre$ and $\dist(J_n, J_m) \ge \vre$.
Then for any $\WW \in \{ \F, \G\}^\N$,
$$\dim_H(C(\WW)) = \underline{\dim}_B(C(\WW)) = \overline{\dim}_B(C(\WW)) 
= \dfrac{-\log k}{\phantom{-} \log \rho}.$$
\end{exa}

\begin{proof}
First note that for each $j \in \N$, we can cover $C(\WW)$ with $k^j$ intervals of length $\rho^j$. It follows
that
$$\dim_H C(\WW) \le \underline{\dim}_B C(\WW) \le \overline{\dim}_B C(\WW)
 \le \dfrac{-\log k}{\phantom{-} \log \rho}$$
(See for example Proposition 4.1 in \cite{F}.)
On the other hand, since the basic intervals at stage-$j$ are separated by gaps with length at
least $\rho^{j-1}\vre$, it follows from an application of the mass distribution principle that
$$\dim_H C(\WW) \ge \liminf_{j \to \infty} \dfrac{\log(k^{j-1})}{-\log(k\rho^{j-1} \vre)} 
= \dfrac{-\log k}{\phantom{-}\log \rho}.$$
(See Example 4.6 in \cite{F}.)
\end{proof}

Thus, if $\F$ and $\G$ satisfy the hypotheses of
Example \ref{dims-equal} as well as those of
either Theorem \ref{gen-bdd-deriv} or Theorem \ref{gen-endpts}, 
then $C(\WW)$ is not equal to $C(\h)$ for any $C^{1+\beta}$-diffeomorphic IFS $\h$ even though
the Hausdorff and box-counting dimensions of $C(\WW)$ agree.

\section{Further Questions}
\label{questions}
We remark that the symmetry assumption in Theorem \ref{gen-bdd-deriv} is used only in the proof of 
Claim \ref{init-str-claim}. It is therefore natural to ask whether by means of a different argument
this assumption can be removed.
\begin{question}
Does Theorem \ref{gen-bdd-deriv} hold if the symmetry assumption is removed?
\end{question}
Another natural direction is to increase the number of IFSs.
\begin{question}
Does an analog of Theorem \ref{gen-bdd-deriv} (or Theorem \ref{gen-endpts}) hold if we replace the two IFSs 
$\F$ and $\G$ with
	\begin{itemize}
		\item a finite number of IFSs $\F_1, \dots, \F_k$?
		\item a sequence of IFSs $\F_1, \F_2, \dots$?
	\end{itemize}
\end{question}

\bibliographystyle{alpha}

\begin{thebibliography}{22}


\bibitem{B} R.\ Bowen, \textsl{Equilibrium States and the Ergodic Theory of Anosov Diffeomorphisms}, 
Lecture Notes in Mathematics, {\bf 470}, Springer-Verlag, New York, 1975. 

\bibitem{Can} S.\ Cantat, \textsl{Bers and H{\' e}non, Painlev{\' e} and Schr{\" o}dinger}, Duke Math.\ J.\ {\bf 149} (2009), no.\ 3, 411--460.

\bibitem{Cas} M.\ Casdagli, \textsl{Symbolic dynamics for the renormalization map of a quasiperiodic Schr{\" o}dinger Equation}, Commun.\ Math.\ Phys.\ 
				{\bf 107}, 295--318 (1986).
				
\bibitem{DG} D.\ Damanik and A.\ Gorodetski, \textsl{Spectral and quantum dynamical properties of the weakly coupled Fibonacci Hamiltonian},
						Comm.\ Math.\ Phys.\ {\bf 305} 221-277 (2011).
						
\bibitem{DKN1} B.\ Deroin, V.\ Kleptsyn, A.\ Navas, \textsl{Sur la dynamique unidimensionnelle en r{\' e}gularit{\' e}
interm{\' e}diaire}, Acta Math. {\bf 199} (2007), no. 2, 199--262.

\bibitem{DKN2} B.\ Deroin, V.\ Kleptsyn, A.\ Navas, \textsl{On the ergodic theory of free group actions by
real-analytic circle diffeomorphisms}, preprint, {\tt arXiv:1312.4133}.

\bibitem{F} K.\ Falconer, \textsl{Fractal Geometry: Mathematical Foundations and Applications}, Wiley, 1990.

\bibitem{Gir} A.\ Girand, \textsl{Dynamical Green functions and discrete Schr{\" o}dinger operators with potentials generated by primitive invertible substitution},
							Nonlinearity, {\bf 27} (2014), no.\ 3, 527--543.

\bibitem{LPW} Q.-H.\ Liu, J.\ Peyri{\`e}re, and Z.-Y.\ Wen, \textsl{Dimension of the spectrum of one-dimensional discrete Schr{\" o}dinger operators with Sturmian potentials}, C.\ R.\ Math.\ Acad.\ Sci.\ Paris {\bf 345} (2007), 667--672.

\bibitem{Mei} M.\ Mei, \textsl{Spectra of discrete Schr{\" o}dinger operators with primitive invertible substitution potentials}, preprint (arXiv:1311.0954).

\bibitem{McD} D.\ McDuff, \textsl{$C^1$-minimal subsets of the circle},
				Ann.\ Inst.\ Fourier (1981), 177-193.
				
\bibitem{MH}
M.~Morse and G.~A.~Hedlund,
\textsl{Symbolic dynamics II. Sturmian trajectories},  Amer. J. Math. {\bf 62} (1940), 1-42

\bibitem{Su} D.\ Sullivan, \textsl{Differentiable structures on fractal-like sets, determined by intrinsic
		scaling functions on dual Cantor sets}, in The Mathematical Heritage of Herman Weyl (Durham, NC, 1987),
		15--23, Proc. Sympos. Pure Math. 48, A.M.S., Providence, RI (1988)

\bibitem{Tak} Y.\ Takahashi, \textsl{Quantum and Spectral Properties of the Labyrinth Model}, preprint, {\tt arXiv:1601.01284}.

\end{thebibliography}

\end{document}